\documentclass[12pt,leqno]{amsart}

\usepackage{amscd}
\usepackage{amssymb}
\usepackage{amsmath}
\usepackage{latexsym} 

\newtheorem{theor}{Theorem}[section]{\bf}{\it}
\newtheorem{lem}[theor]{Lemma}{\bf}{\it}
\newtheorem{cor}[theor]{Corollary}{\bf}{\it}
\newtheorem{rem}{Remark}[section]{\bf}{\rm}
\newtheorem{exam}{Example}[section]{\bf}{\rm}
{\bf}{\rm}

\numberwithin{equation}{section}
\numberwithin{theor}{section}
\numberwithin{rem}{section}

\begin{document}
\title[]{Some measure-theoretic properties of U-statistics applied in statistical physics}
\author{}
\address{}
\email{}
\thanks{}
\author{}
\curraddr{}
\email{}
\thanks{}
\thanks{}

\begin{abstract}
This paper investigates the relationship between various measure-theoretic properties of U-statistics with fixed sample size $N$ and the same properties of their kernels.  Specifically, the random variables are replaced with elements in some 
measure space $(\Lambda; dx)$, the resultant real-valued functions on $\Lambda^N$ being called generalized $N$-means.  It is shown that  a.e. convergence of sequences, measurability, essential boundedness and, under certain conditions, integrability with respect to probability measures of generalized $N$-means and their kernels are equivalent.  These results are crucial for the solution of the inverse problem in classical statistical mechanics in the canonical formulation.  
\end{abstract}

\maketitle

\section{Introduction\label{intro}}

 
Let $(\Lambda; dx)$ be a complete $\sigma $-finite measure space with non-zero measure $dx$, and let  
$d^{k}x$ be the completion of the product measure 
$dx^{\otimes k}$ on $\Lambda^{k}$ for $k\in\Bbb{N}$.  
If $1\leq m\leq N$ are integers, and $u:\Lambda^{m}\rightarrow \Bbb{R}$ is a function, the generalized 
$N$-mean of order $m$ with kernel $u$ is defined in this paper as  
\begin{equation}
\label{GenMean_def}
\left(G_{N,m}u\right)(x_1,...,x_N)=
{\binom{N}{m}}^{-1}
\sum_{1\leq i_{1}<\cdots <i_{m}\leq N}
u(x_{i_{1}},...,x_{i_{m}}).  
\end{equation}
In the following, we investigate whether various measure-theoretic properties of the kernels (such as a.e. convergence of sequences, measurability, essential boundedness, and integrability with respect to probability 
measures) can be deduced from the analogous properties 
of the generalized means, and vice versa.  
Such questions arise in proving the inverse conjecture of statistical physics in the canonical formulation \cite{ChChLi84,Navr15}.  The inverse hypothesis states that 
there is a unique kernel $u$ such that the potential of the form (\ref{GenMean_def}) (traditionally, without the scaling factor) produces a given $m$-particle density.  Technically, the kernel $u$ is sought as a minimizer of 
a relative entropy functional.  
(Strictly speaking, a maximizer of its negative.)  In the canonical formulation, when $m\geq 2$, the a.e. convergence and integrability are first proved for a sequence 
of generalized $N$-means and its a.e. limit, respectively \cite{Navr15}.  Hence the need for the above mentioned equivalence. 

When $m=1$, the inverse conjecture lies at the foundation of density functional theory (DFT) of  inhomogeneous fluids \cite{Evans79}.  Its validity for $m\geq 2$, combined with the DFT approach,  leads to new results in liquid state theory. (This will be explored in a later paper.)  The inverse conjecture for $m\geq 2$ is also customarily assumed in coarse-grained modeling \cite{Noid13,LyLa95}.  
  
Even though the questions addressed in this paper originated in statistical physics, the results may have broader value.  Incidentally, if 
$(x_1,...,x_N)$ are replaced  with random variables, the generalized $N$-mean in (\ref{GenMean_def}) becomes a U-statistic.  U-statistics were introduced by Hoeffding as unbiased estimators of regular functionals \cite{Hoeffding48}.  Since then, they have been extensively 
studied, and numerous applications have been found for them \cite{Koroljuk94, Lee90}.  In common statistical usage, the kernel $u$ is given, and the limit properties of U-statistics are studied as sampling size $N$ goes to infinity.  These issues are different from the ones dealt with in the inverse problem.  Nevertheless, it should be expected that some applications require a measure-theoretic setting in which the properties of U-statistics provided here are useful.   

The paper is organized as follows.  In Section \ref{sec_a.e.conv}, it is proved that a sequence of generalized 
$N$-means converges a.e. on $\Lambda^N$ if and only if 
the corresponding sequence of their kernels converges 
a.e. on $\Lambda^m$ (Theorem \ref{Theor1}).  Despite its apparent simplicity, the 
"only if" part of this statement is not easy to verify.  The difficulty lies in the fact that convergence holds only a.e. on 
$\Lambda^N$.  This can be illustrated on a simple example.     
Suppose that $G_{N,1} u_n\rightarrow U$ everywhere on 
$\Lambda^N$.  Then, for every $x\in\Lambda$, $u_n(x)\rightarrow U(x,...,x)$.  However, the same approach cannot 
be used if the convergence holds only a.e. because the diagonal may be (and often is) a set of measure zero.

The equivalence of measurability and essential boundedness of generalized means and their kernels is established in Theorem \ref{Theor2}.  Section 
\ref{sec_integ} is 
concerned with integrability issues.  The general problem is 
as follows.  A symmetric probability density $P$ on 
$\Lambda^N$ induces a marginal symmetric probability 
density $\rho^{(m)}$ on $\Lambda^m$ upon integrating $P$ 
with respect to any set of $N-m$ variables.  If $1\leq m<N$  and $1\leq r<\infty$, is it true that a generalized $N$-mean 
of order $m$ is in $L^{r}(\Lambda^N; Pd^{N}x)$ if and only if 
its kernel is in $L^{r}(\Lambda^m; \rho^{(m)}d^{m}x)$?  While the "if" part of this question is easy to verify, the "only if" part does not hold in general (Example \ref{Exam_int}).  However, an extra condition on $P$ given in  Theorem \ref{Theor3} ensures that the answer to the above question is positive.  This condition holds in some arbitrarily 
small perturbations in $L^{1}(\Lambda^{N}; d^{N}x)$ (and in 
$L^{\infty}(\Lambda^{N}; d^{N}x)$, if measure $dx$ is finite, and $P$ is essentially bounded) of any symmetric probability density (Theorem \ref{prob_est}).    
 
We will finish this section by introducing some terminology that will be used throughout the paper.   

Subsets of $d^{k} x$ measure zero will be called \textit{null} 
sets, and their complements \textit{co-null}.  The wording "a.e.," "null set," "co-null set" will always be understood relative to the measure $d^{k}x$, with $k$ obvious from the context, and the same is true regarding the measurability of functions.  
If a set $E\subset \Lambda^k$ is measurable, its measure  
will be denoted $\left | E \right | $.  

We will also use the following definitions.  Let $1\leq m < N$ be integers.  Then for any  
$(x_{m+1},...,x_{N})\in \Lambda^{N-m}$ and any $E\subset \Lambda^N$, the $(x_{m+1},...,x_{N})$-section of $E$ is
\begin{equation}
\label{sect_def}
E_{x_{m+1},...,x_{N}}:=\{(x_1,...,x_m)\in\Lambda^m : 
(x_1,...,x_N)\in E\}.
\end{equation}
To shorten the presentation, it is useful to define the operator $\hat{B}_{N,m}$, transforming  
a set $E_m\subset \Lambda^m$ into a set $\hat{B}_{N,m}
E_m\subset \Lambda^N$:  
\begin{equation}
\begin{split}
\label{hatB_def}
&\hat{B}_{N,m}E_{m}:=\\
&\{(x_{1},...,x_{N})\in \Lambda
^{N}:(x_{i_{1}},...,x_{i_{m}})\in E_{m}~\forall~1\leq i_{1}<\cdots <i_{m}\leq N\}.
\end{split}
\end{equation}
Finally, for ease of future argument, we need to extend the definition of $G_{N,m} u$ to the case where $0=m\leq N$, 
and $u\equiv c\in\Bbb{R}$.  In this case, we define $G_{N,0} u=c$.

It should be emphasized that, with the exception of Section 4, the kernel $u$ in (\ref{GenMean_def}) is not assumed to be symmetric, as is customarily done for U-statistics.  

 
\section{Almost everywhere convergence\label{sec_a.e.conv}}

In this section we address the question of whether the a.e. 
convergence of a sequence of generalized means $G_{N,m}u_n$ implies the a.e. convergence of their kernels $u_n$.  In particular, if this is true, then the a.e. limit of a sequence of generalized $N$-means is a generalized 
$N$-mean.    

The following two simple lemmas will be crucial for the 
development of all subsequent arguments.  
\begin{lem}
\label{Lem1}  Let $1\leq k\leq N$ be integers, and let $T_{k}$ be a co-null subset
of $\Lambda ^{k}.$  Then, the set $T_N=\hat{B}_{N,k} T_k$ 
is co-null in $\Lambda ^{N}.$
\end{lem}
\begin{proof}
\begin{equation}
T_N=\bigcap_{1\leq i_{1}<\cdots <i_{k}\leq N} T_{i_{1},...,i_{k}},
\end{equation}
where 
$T_{i_{1},...,i_{k}}:=\{(x_{1},...,x_{N})\in \Lambda^{N}:(x_{i_{1}},...,x_{i_{k}})\in T_{k}\}$.  Therefore,
\begin{equation}
\left | \Lambda^N\setminus T_N\right |\leq 
\sum_{1\leq i_{1}<\cdots <i_{k}\leq N} 
\left | \Lambda^N \setminus T_{i_{1},...,i_{k}}\right |=\binom{N}{k} \left |\left (\Lambda^k\setminus T_k \right )\otimes 
\Lambda^{N-k} \right |=0.
\end{equation}
\end{proof}

\begin{lem}
\label{Lem2}
For any integers $0\leq k\leq m\leq N$, $G_{N,m} G_{m,k}=
G_{N,k}$.  
\end{lem}
\begin{proof}
If $0=k\leq m\leq N$, and $u\equiv c\in \Bbb{R}$ then,
$G_{N,0}u=c$ by definition.  On the other hand, 
$G_{N,m}G_{m,0}u=G_{N,m} c = c$. 

Suppose that $1\leq k\leq m\leq N$, and let $u:\Lambda^k\rightarrow \Bbb{R}$ be any function.  By a simple 
combinatorial argument 
\begin{equation}
\begin{split}
&\left(G_{N,m}G_{m,k} u\right )(x_1,...,x_N)=
\binom{N}{m}^{-1}\binom{m}{k}^{-1}\binom{N-k}{m-k}
\sum_{1\leq i_1<\cdots<i_k\leq N} u(x_{i_1},...,x_{i_k})=\\
&\binom{N}{k}^{-1}\sum_{1\leq i_1<\cdots<i_k\leq N} u(x_{i_1},...,x_{i_k})=G_{N,k}(x_1,...,x_N).
\end{split}
\end{equation}
\end{proof}
The first equality follows because for every $1\leq i_1<\cdots<i_k\leq N$, the term $u(x_{i_1},...,x_{i_k})$ appears 
exactly $\binom{N-k}{m-k}$ times.  
\begin{theor}
\label{Theor1}
Let $1\leq m\leq N$ be integers,and let $(u_n)$ be a 
sequence of finite functions on $\Lambda^m$.  Then the 
following is true.  

There is a finite function $U$ on $\Lambda^N$ such that 
$G_{N,m} u_n\rightarrow U$ a.e. if and only if there is 
a finite function $u$ on $\Lambda^m$ such that 
$u_n\rightarrow u$ a.e.  
\end{theor}
\begin{rem}
\rm{The statement is still true with $U$, $u$ and $u_n$ are a.e. finite.  However, we chose them to be everywhere finite to avoid cluttering the proof with non-essential details.}
\end{rem}
\begin{proof}[Proof of Theorem \ref{Theor1}]
Since there is nothing to prove when $m=N$, we will assume that $1\leq m<N$.  
Suppose first that there exists a finite $u$ such that $u_n\rightarrow u$ on some co-null set 
$E_m\subset \Lambda^m$.  Then the set 
$E_N=\hat{B}_{N,m}E_m\subset \Lambda^N$ is co-null by Lemma \ref{Lem1}.  Moreover, 
$G_{N,m} u_n\rightarrow G_{N,m} u$ on $E_N$. 

Conversely, suppose that there exists a finite $U$ such that $G_{N,m} u_n\rightarrow U$ on some co-null set 
$E\subset \Lambda^N$.   

\textit{Case 1: $m=1$}.  Let us fix 
$(\tilde{x}_2,...,\tilde{x}_N)\in \Lambda^{N-1}$ such that 
$\left | \Lambda\setminus E_{\tilde{x}_2,...,\tilde{x}_N} \right | =0$.  
(This is possible because the a.e. section of a co-null 
set is co-null by the Fubini-Tonelli theorem \cite[Theorem 2.39]{Folland99}.)  By the definition of the set $E$, for every 
$x\in E_{\tilde{x}_2,...,\tilde{x}_N}$:
\begin{equation}
\label{a.e.conv_1}
u_n(x)+\sum_{j=2}^{N} u_n(\tilde{x}_j)\rightarrow N 
U(x, \tilde{x}_2,...,\tilde{x}_N).
\end{equation}
Thus, for any $(y_1,...,y_N)\in 
E_{\tilde{x}_2,...,\tilde{x}_N}^N $ :
\begin{equation}
\label{a.e.conv_2}
\sum_{i=1}^{N} u_n(y_i) + N\sum_{j=2}^{N} u_n(\tilde{x}_j)\rightarrow N \sum_{i=1}^{N}U (y_i,\tilde{x}_2,...,\tilde{x}_N),
\end{equation}
where we have summed (\ref{a.e.conv_1}) over $i$ after 
replacing $x$ with $y_i$.  Let us fix $(y_1,...,y_N)\in 
E_{\tilde{x}_2,...,\tilde{x}_N}^N\cap E $.  Then, 
(\ref{a.e.conv_2}) holds together with 
$\sum_{i=1}^{N} u_n(y_i) \rightarrow N U(y_1,...,y_N)$.  Thus, 
\begin{equation}
\label{a.e.conv_3}
\sum_{i=2}^{N} u_n(\tilde{x_i}) \rightarrow 
\sum_{i=1}^{N} U(y_i,\tilde{x}_2,...,\tilde{x}_N)-
U(y_1,...,y_N)=: C.
\end{equation}
Using this result in (\ref{a.e.conv_1}), we finally obtain:
\begin{equation}
u_n(x)\rightarrow N U(x, \tilde{x}_2,...,\tilde{x}_N)-C\quad 
\forall x\in E_{\tilde{x}_2,...,\tilde{x}_N}.
\end{equation}

\textit{Case 2: $2\leq m<N$}.  The proof for this case 
will proceed by induction on $m$.  Let us define 
$M:=\min(m,N-m)$.  Suppose that 
the "only if" 
statement of the theorem is true for $m-1$.  By the Fubini-Tonelli theorem, we can fix 
$(\tilde{x}_{m+1},...,\tilde{x}_N)\in \Lambda^{N-m}$ such that 
$\left | \Lambda^m\setminus 
E_{\tilde{x}_{m+1},...,\tilde{x}_N} \right | =0$.  By the definition of the set $E$:  
\begin{equation}
\label{a.e.convm_1}
\left (G_{N,m} u_n\right )(\cdot,\tilde{x}_{m+1},...,\tilde{x}_N)\rightarrow \\ U(\cdot,\tilde{x}_{m+1},...,\tilde{x}_N)~~\text{
on $E_{\tilde{x}_{m+1},...,\tilde{x}_N}$}.
\end{equation}
Using the functions $v_n^{(m-k)}:\Lambda^{m-k}\rightarrow \Bbb{R}$, $1\leq k\leq m-1$, and constants $v_n^{(0)}\in\Bbb{R}$ defined as:
\begin{equation}
\begin{split}
&v_n^{(m-k)}:=\\
&\sum_{1\leq j_1<\cdots<j_k\leq N-m} 
u_n(\cdot,\tilde{x}_{m+j_1},...,\tilde{x}_{m+j_k})\quad
\text{if}~~1\leq k\leq m-1, \\
&v_n^{(0)}:=\sum_{1\leq j_1<\cdots<j_m\leq N-m} 
u_n(\tilde{x}_{m+j_1},...,\tilde{x}_{m+j_m}),
\end{split}
\end{equation}
the left hand side of (\ref{a.e.convm_1}) can be rewritten as 
\begin{equation}
\label{a.e.convm_2}
\begin{split}
\binom{N}{m}
&\left (G_{N,m} u_n\right )(\cdot,\tilde{x}_{m+1},...,\tilde{x}_N)
=\\
&\left [ u_n + \sum_{k=1}^{M}\binom{m}{m-k} 
G_{m,m-k} v_n^{(m-k)}\right ].
\end{split}
\end{equation}
By Lemma \ref{Lem2}, $G_{m,m-k}=G_{m,m-1}G_{m-1,m-k}$, $1\leq k\leq m$.  Thus, the right hand side of (\ref{a.e.convm_2}) simplifies to:
$u_n + \left (G_{m,m-1} \omega_n\right )$,
where $\omega_n: \Lambda^{m-1}\rightarrow \Bbb{R}$ is 
\begin{equation}
\label{omega_n}
\omega_n:=\sum_{k=1}^{M}\binom{m}{m-k} 
G_{m-1,m-k} v_n^{(m-k)}.
\end{equation}
Then, in view of (\ref{a.e.convm_1}), we obtain that 
\begin{equation}
\label{a.e.convm_3}
u_n + G_{m,m-1} \omega_n 
\rightarrow \binom{N}{m}
U(\cdot,\tilde{x}_{m+1},...,\tilde{x}_N)~~ \text{on 
$E_{\tilde{x}_{m+1},...,\tilde{x}_N}$}.
\end{equation}
The set $\Omega= \hat{B}_{N,m} 
E_{\tilde{x}_{m+1},...,\tilde{x}_N}\subset \Lambda^N $ is co-null by Lemma \ref{Lem1}.  Moreover, 
applying the operator $G_{N,m}$ to both sides of 
(\ref{a.e.convm_3}), gives that on $\Omega$: 
\begin{equation}
\begin{split}
&G_{N,m} u_n  + G_{N,m} G_{m,m-1} \omega_n = \\
& G_{N,m} u_n + G_{N,m-1}  \omega_n 
\rightarrow \binom{N}{m}  G_{N,m} 
U(\cdot,\tilde{x}_{m+1},...,\tilde{x}_N).
\end{split}
\label{a.e.convm_4}
\end{equation}
(Lemma \ref{Lem2} was used in the equality.)  
Since (\ref{a.e.convm_4}) and $G_{N,m} u_n \rightarrow 
U$ both hold on $\Omega\cap E$, we further obtain that
\begin{equation}
\label{a.e.convm_5}
G_{N,m-1}  \omega_n 
\rightarrow \binom{N}{m}  G_{N,m}U(\cdot,\tilde{x}_{m+1},...,\tilde{x}_N) - U~~\text{on $\Omega\cap E$}.
\end{equation}
Now, the induction hypothesis implies that there is a finite 
function $\omega:\Lambda^{m-1}\rightarrow \Bbb{R}$ such 
that $\omega_n\rightarrow \omega$ a.e.  Further, by the "if" statement of the theorem there is a co-null set $E_m\subset
\Lambda^m$ and a finite 
function $\phi:\Lambda^{m}\rightarrow \Bbb{R}$ such 
that $G_{m,m-1}\omega_n\rightarrow \phi$ on $E_m$.  Using this 
result in (\ref{a.e.convm_3}), we finally obtain:
\begin{equation}
\label{a.e.convm_6}
u_n\rightarrow \binom{N}{m}
U(\cdot,\tilde{x}_{m+1},...,\tilde{x}_N)-
\phi~~ \text{on 
$E_m\cap E_{\tilde{x}_{m+1},...,\tilde{x}_N}$}.
\end{equation}

\end{proof}

\begin{cor}
\label{cor1}
Let $1\leq m \leq N$ be integers, and $u$ be a finite 
function on $\Lambda^m$.  Then, $G_{N,m}u=0$ a.e. if and only if $u=0$ a.e.  In particular, the linear operator 
$G_{N,m}$ is injective.  
\end{cor}
\begin{proof}
It is easy to verify, using Lemma \ref{Lem1}, that $G_{N,m}u=0$ a.e. if $u=0$ a.e.  For the converse, let us 
define a sequence $(v_n)$ by $v_n=u$ if $n$ is odd and $v_n=0$ if 
$n$ is even.  Since $G_{N,m} v_n\rightarrow 0 $ a.e., Theorem \ref{Theor1} implies that there is a finite function 
$v$ on $\Lambda^m$ such that $v_n\rightarrow v $ a.e.  Then, 
$v=u=0$ a.e. by the definition of sequence $(v_n)$.  
\end{proof}

\section{Measurability and essential boundedness\label{sec_measur}}

\begin{theor}
\label{Theor2}
Let $1\leq m\leq N$ be integers, and 
$u:\Lambda^m\rightarrow \Bbb{R}$ be a function.  
Let $U:=G_{N,m}u$.  
Then,
\newline
(i) $U$ is measurable if and only if $u$ is measurable.\newline
(ii) $U\in L^{\infty}(\Lambda^N; d^{N}x)$ if and only if 
$u\in L^{\infty}(\Lambda^m; d^{m}x)$.  Moreover, there is a 
constant $C(N,m)$ such that
$|| U ||_{\infty,d^{N}x}\leq || u ||_{\infty,d^{m}x}\leq 
C(N,m) || U ||_{\infty,d^{N}x}$, with 
$C(N,1)=1$.  In particular, 
$G_{N,m}$ is an isomorphism from $L^{\infty}(\Lambda^m; d^{m}x)$ onto a closed subspace of 
$L^{\infty}(\Lambda^N; d^{N}x)$.
\end{theor}
\begin{proof}
"In particular" part is the direct consequence of Theorem \ref{Theor1} and Corollary \ref{cor1}.
The conclusion holds trivially for $m=N$, so let us assume that $1\leq m<N$.  Since the proofs of (i) and (ii) are very 
similar, we will only show (ii).  However, it will be clear that (i) 
is established by a shorter version of the same argument.    
 
Suppose first that $u\in L^{\infty}(\Lambda^m; d^{m}x)$.    Then, $U$ is a finite sum of measurable functions.  
Namely, $U=\sum_{1\leq i_1<\cdots<i_m\leq N} 
U_{i_1,...,i_m}$, where $U_{i_1,...,i_m}(x_1,...,x_N):=
u(x_{i_1},...,x_{i_m})$.  Moreover, the set 
$E_m\subset\Lambda^m$ on which $| u | \leq || u ||_{\infty,d^{m}x}$ is co-null.  It follows that $\left | U \right | \leq || u ||_{\infty,d^{m}x} $
on the co-null set $\hat{B}_{N,m} E_m\subset \Lambda^N$, and so  $\left |\left | U \right |\right |_{\infty,d^{N}x} \leq 
|| u ||_{\infty,d^{m}x} $.

Conversely, suppose that $U\in L^{\infty}(\Lambda^N; d^{N}x)$.  
\newline

\textit{Case 1}:  $m=1$.  The Fubini-Tonelli 
theorem implies that there is $(\tilde{x}_2,...,\tilde{x}_N)\in \Lambda^{N-1}$ 
such that $U(\cdot,\tilde{x}_2,...,\tilde{x}_N)=u(\cdot) + 
\sum_{i=2}^{N} u(\tilde{x}_i)\in L^{\infty}(\Lambda; dx)$), and so $u\in L^{\infty}(\Lambda; dx)$. 


Next, we will show that $|| u ||_{\infty,dx}=\left | \left | U 
\right |\right |_{\infty,d^{N}x} $.  In view of the "if" part of the theorem, it suffices to prove that $\left | \left | U \right |\right |_{\infty,d^{N}x} \geq || u ||_{\infty,dx}$.  

Let $s:=$ ess sup $u $, and $S:=$ ess sup $U $.
For every 
$\epsilon>0$, the measure of the set $A_{\epsilon}:=\{x\in\Lambda : u(x)>s-\epsilon\}$ is strictly positive.  Since the 
set $E:=\{(x_1,...,x_N)\in\Lambda^N: U(x_1,...,x_N)\leq S\}
\subset \Lambda^N$ is co-null, it follows that $\left | A_{\epsilon}^N\cap E\right |>0$.  Moreover, for every 
$(x_1,...,x_N)\in A_{\epsilon}^N\cap E$, 
$S\geq U(x_1,...,x_N) > s-\epsilon$.  Thus, $S\geq s$.  Similarly, 
ess inf $U\leq$ ess inf $u$, and so $\left | \left | U \right |\right |_{\infty,d^{N}x} \geq || u ||_{\infty,dx}$.
 
\textit{Case 2}: $2\leq m<N$.  The proof will proceed by  
induction on $m$.  Suppose that the "only if" statement of the theorem holds for $m-1$.  
The Fubini-Tonelli theorem implies that there is $(\tilde{x}_{m+1},...,\tilde{x}_{N})\in\Lambda^{N-m}$ such that 
$U(\cdot,\tilde{x}_{m+1},...,\tilde{x}_{N})\in 
L^{\infty}(\Lambda^m; d^{m}x)$, with 
\begin{equation}
\label{ineq_1}
|| U(\cdot,\tilde{x}_{m+1},...,\tilde{x}_{N}) ||_{\infty,d^{m}x}\leq || U ||_{\infty,d^{N}x}. 
\end{equation} 
By the proof of Theorem \ref{Theor1} (see (\ref{a.e.convm_1} - \ref{a.e.convm_3}) ), there is $\omega:\Lambda^{m-1}\rightarrow \Bbb{R}$ such that 
\begin{equation}
\label{equality_1}
\binom{N}{m} U(\cdot,\tilde{x}_{m+1},...,\tilde{x}_{N})=
u + G_{m,m-1} \omega .  
\end{equation}
Applying $G_{N,m}$ to both sides of (\ref{equality_1}), and 
using Lemma \ref{Lem2} yield:
\begin{equation}
G_{N,m-1}\omega = \binom{N}{m} 
G_{N,m} U(\cdot,\tilde{x}_{m+1},...,\tilde{x}_{N})-U.
\end{equation}
Therefore, by the "if" statement of the theorem and 
(\ref{ineq_1}), $G_{N,m-1}\omega\in 
L^{\infty}(\Lambda^N; d^{N}x)$, with 
\begin{equation}
\label{ineq_2}
\begin{split}
&\left | \left |G_{N,m-1}\omega\right | \right |_{\infty,d^{N}x}
\leq \binom{N}{m} || U(\cdot,\tilde{x}_{m+1},...,\tilde{x}_{N}) 
||_{\infty,d^{m}x} + || U ||_{\infty,d^{N}x} \\
&\leq 
\left ( \binom{N}{m} +1 \right ) || U ||_{\infty,d^{N}x}.
\end{split}
\end{equation}
Then, the induction hypothesis implies that 
$\omega\in 
L^{\infty}(\Lambda^{m-1}; d^{m-1}x)$, and  
\begin{equation}
\label{ineq_3}
||\omega ||_{\infty,d^{m-1}x}\leq C(N,m-1)
\left ( \binom{N}{m} +1 \right ) || U ||_{\infty,d^{N}x}
\end{equation}
for some constant $C(N,m-1)$.  Next, using the "if" part of 
the theorem again, we infer that $G_{m,m-1}\omega\in
L^{\infty}(\Lambda^m;d^{m}x)$, and $\left | \left | 
G_{m,m-1}\omega \right | \right |_{\infty,d^{m}x}\leq 
||\omega ||_{\infty,d^{m-1}x}$.  This result, together 
with (\ref{equality_1}), (\ref{ineq_1}), and (\ref{ineq_3}),  
yield that $u\in L^{\infty}(\Lambda^m;d^{m}x)$, and
\begin{equation}
\begin{split}
&|| u ||_{\infty,d^{m}x}\leq \binom{N}{m} 
\left | \left | U(\cdot,\tilde{x}_{m+1},...,\tilde{x}_{N}) 
\right |\right |_{\infty,d^{m}x} + 
||\omega ||_{\infty,d^{m-1}x}\\
& \leq \left [ \binom{N}{m} \left ( 1 + C(N,m-1)\right ) + 
C(N,m-1) \right ] 
\left | \left | U 
\right |\right |_{\infty,d^{N}x}.
\end{split}
\end{equation}
\end{proof}

\section{Integrability\label{sec_integ}}

Let $N\geq 2$ and $P$ be a symmetric probability density 
on $\Lambda^N$.  That is, $P$ is a nonnegative, symmetric function, and $\int_{\Lambda^N} P=1$.  For every 
$1\leq m<N$, the marginal probability density 
on $\Lambda^m$ is defined as 
\begin{equation}
\label{rho^m_def}
\rho^{(m)}:=\int_{\Lambda^{N-m}} P(\cdot,x_{m+1},...,x_N)dx_{m+1}\cdots dx_{N}~~\text{a.e. on $\Lambda^m$}.  
\end{equation}
Note that $\rho^{(m)}$ is symmetric a.e. on $\Lambda^m$.  
In this section we will discuss the relationship between integrability of generalized $N$-means with respect to measure $Pd^{N}x$ and integrability of their kernels with respect to measure $\rho^{(m)}d^{m}x$. 

It is easy to convince oneself that $u\in L^{r}(\Lambda^m;
\rho^{(m)} d^{m} x)$ implies that 
$G_{N,m}u\in L^{r}(\Lambda^N;P d^{N} x)$ for 
$1\leq r<\infty$.  However, the converse 
is not obvious and, in fact, is not true in general.  This situation is illustrated with the following example.  

\begin{exam}
\label{Exam_int}
\rm{
Consider a $\sigma$-finite measure space $(\Lambda; d\mu)$, where $\Lambda=\Bbb{N}$, and $d\mu$ is the counting measure.  Let us define the probability density $P$ on 
$\Lambda^2$ by the formula:
\begin{equation}
\label{ex_P}
P(i,j):=\left \{
\begin{array}{ll}
\frac{1}{(i+j)^2}&\text{if $|i-j|=1$}\\
0&\text{if $|i-j|\neq1$}
\end{array}
\right .
\end{equation}
Then, $P$ is symmetric, and 
\begin{equation}
\int_{\Lambda^2} P d^{2}\mu = 2 \sum_{i=1}^{\infty} 
\frac{1}{(2i+1)^2}<\infty.
\end{equation}
(That $\int_{\Lambda^2} P d^{2}\mu\neq 1$ is immaterial.)  
For every $i\geq 2$, the marginal probability density 
$\rho^{(1)}(i)$ is calculated to be 
\begin{equation}
\label{ex_rho1}
\rho^{(1)}(i)=\sum_{j=1}^{\infty} P(i,j)=
\frac{1}{(2i+1)^2}+\frac{1}{(2i-1)^2}>\frac{1}{(2i+1)^2}.
\end{equation}
Let us define $u:\Lambda\rightarrow \Bbb{R}$ by 
$u(i)=2(-1)^{i}i$.  Then, $|u(i)+u(j)|=2$ whenever 
$|i-j|=1$.  Thus, using (\ref{ex_P}) and (\ref{ex_rho1}), we find that
\begin{equation}
\int_{\Lambda^2} P |G_{2,1} u |d^{2}\mu = 
2 \sum_{i=1}^{\infty} \frac{1}{(2i+1)^2}<\infty, 
\end{equation}
but
\begin{equation}
\int_{\Lambda}\rho^{(1)}|u| d\mu>2\sum_{i=2}^{\infty}
 \frac{i}{(2i+1)^2}=\infty.
\end{equation} 
 }
\end{exam} 
In spite of Example \ref{Exam_int}, an extra condition on $P$ ensures that a generalized $N$-mean of order $m$ is in 
$L^{1}(\Lambda^{N};Pd^{N}x)$ if and only if its kernel is in 
$L^{1}(\Lambda^{m};\rho^{(m)}d^{m}x)$.  The generality of 
this condition is addressed in Theorem \ref{prob_est}.  
We begin with two lemmas that will be used in the proof of 
Theorem \ref{Theor3}, the main result of this section.    

\begin{lem}
\label{Lem_int3}
Let $1\leq m\leq N-1$ be integers, $A\subset \Lambda$ be 
a subset of positive measure, and 
$\gamma: A\rightarrow (0,\infty)$ be a measurable  function.  If, 
\begin{equation}
\label{int_lem3_eq1}
\rho^{(m+1)}(x_1,...,x_{m+1})\geq \gamma(x_{m+1}) 
\rho^{(m)}(x_1,...,x_{m}) 
\end{equation}
for a.e.  
$(x_1,...,x_{m+1})\in \Lambda^m\otimes A$, then 
\begin{equation}
\label{int_lem3_eq2}
\rho^{(m)}(x_1,...,x_{m})\geq \gamma(x_{m}) 
\rho^{(m-1)}(x_1,...,x_{m-1}) 
\end{equation}
for a.e. 
$(x_1,...,x_{m})\in \Lambda^{m-1}\otimes A$.
\end{lem}
\begin{proof}
Let $E_{m+1}\subset\Lambda^{m+1}$ be a co-null set 
such that (\ref{int_lem3_eq1}) holds on 
$E_{m+1} \cap \left (\Lambda^m\otimes A\right )$.  By the 
Fubini-Tonelli theorem, there is a co-null set $E_m\subset 
\Lambda^m$ such that for every 
$(x_1,...,x_{m-1},x_{m+1})\in E_m$ both sides of 
(\ref{int_lem3_eq1}) are integrable functions of $x_m$, and 
the section $(E_{m+1})_{x_1,...,x_{m-1},x_{m+1}}\subset 
\Lambda$ is co-null.  If $(x_1,...,x_{m-1},x_{m+1})\in E_m\cap \left (
\Lambda^{m-1}\otimes A\right )$, and  $x_m\in 
(E_{m+1})_{x_1,...,x_{m-1},x_{m+1}}$, then
$(x_1,...,x_{m+1})\in E_{m+1}\cap \left (\Lambda^m\otimes A\right )$, and so inequality (\ref{int_lem3_eq1}) holds.  
Thus, integrating both sides of (\ref{int_lem3_eq1}) with respect to $x_m$, and  
subsequently renaming $m+1$ with $m$ yield:
\begin{equation}
\rho^{(m)}(x_1,...,x_{m})\geq \gamma(x_{m}) 
\rho^{(m-1)}(x_1,...,x_{m-1}) 
\end{equation}
for every 
$(x_1,...,x_{m})\in E_m\cap \left (\Lambda^{m-1}\otimes A\right )$.
\end{proof}

\begin{lem}
\label{Lem_int4}
Let $1\leq m\leq N-1$ be integers, $A\subset \Lambda$ be 
a subset of positive measure, and 
$\gamma: A\rightarrow (0,\infty)$ be a measurable  
function.  Suppose that 
\begin{equation}
\label{int_lem4_eq1}
P(x_1,...,x_{N})\geq \gamma(x_{N}) 
\rho^{(N-1)}(x_1,...,x_{N-1})
\end{equation}
for a.e.  
$(x_1,...,x_{N})\in \Lambda^{N-1}\otimes A$.  Then, 
\begin{equation}
\label{int_lem4_eq2}
P(x_1,...,x_{N})\geq \gamma(x_{N})\cdots\gamma(x_{m+1}) 
\rho^{(m)}(x_1,...,x_{m}) 
\end{equation}
for a.e. 
$(x_1,...,x_{N})\in \Lambda^{m}\otimes A^{N-m}$.  
\end{lem}

\begin{proof}
Inequality (\ref{int_lem4_eq2}) clearly holds when $m=N-1$.  
Suppose that it is satisfied for some $2\leq m \leq N-1$.  We 
will show that it then holds for $m-1$, and so the lemma will 
follow by induction.  

Using Lemma \ref{Lem_int3}, we infer from (\ref{int_lem4_eq1}) by induction that 
\begin{equation}
\label{int_lem4_eq3}
\rho^{(m)}(x_1,...,x_{m}) \geq \gamma(x_m) 
\rho^{(m-1)}(x_1,...,x_{m-1})
\end{equation}
for a.e. $(x_1,...,x_{m})\in \Lambda^{m-1}\otimes A$.  
Therefore, in view of (\ref{int_lem4_eq2}), 
\begin{equation}
\label{int_lem4_eq4}
P(x_1,...,x_{N})\geq \gamma(x_{N})\cdots\gamma(x_{m}) 
\rho^{(m-1)}(x_1,...,x_{m-1}) 
\end{equation}
for a.e. 
$(x_1,...,x_{N})\in \Lambda^{m-1}\otimes A^{N-m+1}$. 
\end{proof}

\begin{theor}
\label{Theor3}
Suppose $N\geq 2$ is an integer, and that for every $x_N$ in some subset $A\subset \Lambda$ of positive measure, 
there is a constant $\gamma(x_N)>0$ such that 
\begin{equation}
\label{int_cond}
P(\cdot,x_N)\geq \gamma(x_N)\rho^{(N-1)}~~\text{ a.e. on $\Lambda^{N-1}$}.  
\end{equation}
Let $1\leq m\leq N$ be integers, $1\leq r<\infty$, and 
$u:\Lambda^m\rightarrow \Bbb{R}$ be a function.  
Define $U:= G_{N,m} u$.   
Then, \newline
$U\in L^{r}(\Lambda^N; Pd^{N}x)$ if and only if 
$u\in L^{r}(\Lambda^m; \rho^{(m)}d^{m}x)$.  
Moreover, there is a 
constant $C:=C(N,m,r,P)$ such that\newline
$|| U||_{r, P d^{N} x}\leq || u ||_{r,\rho^{(m)} d^{m}x}\leq 
C || U ||_{r, P d^{N} x}$.  In particular, 
$G_{N,m}$ is an isomorphism from $L^{r}(\Lambda^m; d^{m}x)$ onto a closed subspace of 
$L^{r}(\Lambda^N; d^{N}x)$.
\end{theor}
\begin{rem}
\label{rem_int}
\rm{The condition on $P$ at the beginning of Theorem 
\ref{Theor3} can be replaced with another, seemingly stronger, but in fact equivalent, assumption.  To be specific, we can assume that    
\begin{equation}
\label{int_cond_equiv}
P\geq \rho^{(N-1)}\otimes \gamma~~\text{a.e. on $\Lambda^{N-1}\otimes A$},  
\end{equation}
where $A\subset \Lambda$ is some subset of positive measure, and $\gamma: A\rightarrow (0,\infty)$ is 
a measurable function.

Indeed, the condition on $P$ stated in Theorem \ref{Theor3} is equivalent to:
ess inf $f(\cdot, x_N)>0$ for every $x_N\in A$, where $f$ is a measurable function on 
$\Lambda^N$ defined by
\begin{equation}
f=\left \{  
\begin{array}{ll}
P\slash (\rho^{(N-1)}\otimes 1)&\text{if 
$\rho^{(N-1)}\otimes 1>0$,}\\
1&\text{if $\rho^{(N-1)}\otimes 1=0$.}
\end{array}
\right .
\end{equation}
Moreover, (\ref{int_cond}) holds with $\gamma(x_N)$ replaced by ess inf $f(\cdot, x_N)$, a measurable function.    
However, if $\gamma$ in (\ref{int_cond}) is $dx$ measurable, 
then $g:= P-\rho^{(N-1)}\otimes \gamma$ is $d^{N} x$ 
measurable.  Let $T:=\{(x_1,...,x_N)\in \Lambda^{N}: 
g(x_1,...,x_N)\geq 0\}$.  Arguing by contradiction, it is easy 
to see that $|\Lambda^{N-1}\otimes A\setminus T | =0$, i.e. 
(\ref{int_cond_equiv}) holds a.e. on $\Lambda^{N-1}\otimes A$.  

}

\end{rem}
\begin{proof}[Proof of Theorem \ref{Theor3}]
The "in particular" part is the direct consequence of Theorem \ref{Theor1} and Corollary \ref{cor1}.  Since there is nothing to prove when $m=N$, we will assume that $1\leq m <N$.   

For the "if" part of the theorem, suppose that $u\in L^{r}(\Lambda^m; d^{m}x)$.  Then,
\begin{equation}
\begin{split}
&|| U ||_{r, P d^{N}x}^r=\int_{\Lambda^N} |G_{N,m} u |^r P
\leq \\
&\binom{N}{m}^{-1} \sum_{1\leq i_1<\cdots<i_m\leq N }
\int_{\Lambda^{N}} | u(x_{i_1},...,x_{i_m}) |^r P dx_1\cdots dx_N=\\
&\int_{\Lambda^{m}} | u |^r \rho^{(m)} = 
|| u ||_{r, \rho^{(m)} d^{m}x}^r.
\end{split}
\end{equation}

For the "only if" part, we will use the condition on $P$, as stated in Remark \ref{rem_int}.  Suppose that $U\in L^{r}(\Lambda^N; P d^{N}x)$.  Since $\gamma$ is positive on $A$, there is 
$\varepsilon >0$ such that the set 
$A_{\varepsilon}:=\{x\in A : \gamma(x) > \varepsilon \} $ has positive measure.  Thus, in view of Lemma 
\ref{Lem_int4}, for every $1\leq m\leq N-1$:
\begin{equation}
\label{P_rho}
\begin{split}
&P(x_{1},...,x_N)\geq \\
&\gamma(x_N)\cdots\gamma(x_{m+1})
\rho^{(m)}(x_1,...,x_m)>\varepsilon^{N-m} \rho^{(m)}(x_1,...,x_m)
\end{split}
\end{equation}
for a.e. $(x_{1},...,x_N)\in \Lambda^{m}\otimes 
A_{\varepsilon}^{N-m}$.  
Note, that $|A_{\varepsilon}|<\infty$ because the integration 
of (\ref{P_rho}) gives $1\geq \int_{\Lambda^m\otimes 
A_{\varepsilon}^{N-m}} P \geq \left (\varepsilon |A_{\varepsilon}|\right )^{N-m}$.  
\begin{lem}
\label{Lem_int5}
Let $\alpha:=\left (\varepsilon |A_{\varepsilon}|\right )^{-1}$.  
For every $1\leq m\leq N-1$, the set 
\begin{equation}
\label{def_T_N-1}
\begin{split}
&T_{N-m}:=
\{(x_{m+1},...,x_N)\in A_{\varepsilon}^{N-m} : \\
&\text{$U(\cdot,x_{m+1},...,x_N)$ is measurable and} \\
&|| U(\cdot,x_{m+1},...,x_N)||_{r,\rho^{(m)} d^{m}x }^r \leq 
\alpha^{N-m} || U ||_{r,Pd^{N}x }^r \}
\end{split}
\end{equation}
is not a set of measure zero.  In particular, it is not empty.  
\end{lem}
\begin{proof}
Suppose that $| T_{N-m}|=0$.  This means that 
\begin{equation}
\label{case1_contr1}
|| U ||_{r,Pd^{N}x }^r<\alpha^{m-N} 
\int_{\Lambda^m} | U(\cdot,x_{m+1},...,x_N) |^{r} \rho^{(m)}d^{m} x 
\end{equation}
for a.e. $(x_{m+1},...,x_N)\in A_{\varepsilon}^{N-m} $.  Integration 
over $ A_{\varepsilon}^{N-m}$ in the last inequality, and 
(\ref{P_rho}) give:
\begin{equation}
\begin{split}
\label{case1_contr2}
&|| U ||_{r,Pd^{N}x }^r<\varepsilon^{N-m} 
\int_{\Lambda^m\otimes A_{\varepsilon}^{N-m}} | U(x_1,...,x_N) |^{r} \rho^{(m)}(x_1,...,x_m) d^{N}x\\
&<\int_{\Lambda^m\otimes A_{\varepsilon}^{N-m}} | U |^{r}
P d^{N}x\leq|| U ||_{r,Pd^{N}x }^r, 
\end{split}
\end{equation}
a contradiction.  Thus, $T_{N-m}$ can not be a set of measure zero.  
\end{proof}

\textit{Case 1:}  $m=1$.   
Let us fix $(\tilde{x}_2,...,\tilde{x}_N)\in T_{N-1}$, a non-empty set by Lemma \ref{Lem_int5}.   Then,
\begin{equation}
\label{NU_u}
NU(\cdot,\tilde{x}_2,...,\tilde{x}_N)=u + \sum_{i=2}^N 
u(\tilde{x}_i):=u +\tilde{c}.
\end{equation}
By the definition of the set $T_{N-1}$ given in (\ref{def_T_N-1}),
\begin{equation}
\label{Uhro_UP}
|| U(\cdot,\tilde{x}_2,...,\tilde{x}_N)||_{r,\rho^{(1)} dx } \leq 
\alpha^{\frac{N-1}{r}} || U ||_{r,Pd^{N}x} .  
\end{equation}
To get an estimate on $\tilde{c}$, let us apply $G_{N,1}$ to 
both sides of \ref{NU_u}, with the result 
$\tilde{c}=N G_{N,1} U(\cdot,\tilde{x}_2,...,\tilde{x}_N) - U$.  
Then, the "if" part of the theorem and (\ref{Uhro_UP}) 
imply that 
\begin{equation}
\label{c_est}
\begin{split}
|\tilde{c}|&\leq N || G_{N,1} U
(\cdot,\tilde{x}_2,...,\tilde{x}_N)||_{r,Pd^{N}x} + 
|| U ||_{r,Pd^{N}x}\\
&\leq N || U
(\cdot,\tilde{x}_2,...,\tilde{x}_N)||_{r,\rho^{(1)}dx} + 
|| U ||_{r,Pd^{N}x}\\
&\leq \left [ 
N \alpha^{\frac{N-1}{r}}
+1 \right ]
|| U ||_{r,Pd^{N}x}.
\end{split}
\end{equation}
Finally, we infer from (\ref{NU_u}), (\ref{Uhro_UP}), and (\ref{c_est}) that 
\begin{equation}
\begin{split}
&|| u ||_{r,\rho^{(1)}dx} \leq N || U(\cdot,\tilde{x}_2,...,\tilde{x}_N)||_{r,\rho^{(1)}dx} + |\tilde{c}|\\
&\leq \left (
2N \alpha^{\frac{N-1}{r}} 
+1 \right ) || U ||_{r,Pd^{N}x}.  
\end{split}
\end{equation}
Note that the constant in the round brackets depends on 
$P$ through $\alpha$.  

\textit{Case 2:}  $2\leq m < N$.  Similarly to the proofs for this case in the 
previous two theorems, we will use induction on 
$m$.  Suppose that the "only if" statement of the theorem 
holds for $m-1$.  Let us fix $(\tilde{x}_{m+1},...,\tilde{x}_N)\in 
T_{N-m}$, a non-empty set by Lemma \ref{Lem_int5}.  As was 
shown previously, (see (\ref{a.e.convm_1} - \ref{a.e.convm_3}) ), there is $\omega:\Lambda^{m-1}\rightarrow \Bbb{R}$ such that 
\begin{equation}
\label{int_mg2_1}
\binom{N}{m} U(\cdot,\tilde{x}_{m+1},...,\tilde{x}_{N})=
u + G_{m,m-1} \omega.  
\end{equation}
By the definition of the set $T_{N-m}$ given in (\ref{def_T_N-1}),
\begin{equation}
\label{Uhrom_UP}
|| U(\cdot,\tilde{x}_{m+1},...,\tilde{x}_N)||_{r,\rho^{(m)} d^{m}x } \leq 
\alpha^{\frac{N-m}{r}} || U ||_{r,Pd^{N}x} .  
\end{equation}
An estimate on $|| G_{m,m-1}\omega ||_{r,\rho^{(m)} d^{m}x }$ will follow by induction.  Applying $G_{N,m}$ to both 
sides of (\ref{int_mg2_1}) and using Lemma \ref{Lem2} yield:
\begin{equation}
G_{N,m-1}\omega = \binom{N}{m} 
G_{N,m} U(\cdot,\tilde{x}_{m+1},...,\tilde{x}_{N})-U.
\end{equation}
The last equation shows that $G_{N,m-1}\omega$ is 
measurable.  In addition, using the "if" statement of the 
theorem and (\ref{Uhrom_UP}), we estimate:
\begin{equation}
\label{Gw_est}
\begin{split}
||G_{N,m-1}\omega||_{r,Pd^{N}x}&\leq \binom{N}{m} 
|| G_{N,m} U
(\cdot,\tilde{x}_{m+1},...,\tilde{x}_N)||_{r,Pd^{N}x} + 
|| U ||_{r,Pd^{N}x}\\
&\leq \binom{N}{m} || U
(\cdot,\tilde{x}_{m+1},...,\tilde{x}_N)||_{r,\rho^{(m)}d^{m}x} + 
|| U ||_{r,Pd^{N}x}\\
&\leq \left [ 
\binom{N}{m} \alpha^{\frac{N-m}{r}}
+1 \right ]
|| U ||_{r,Pd^{N}x}.  
\end{split}
\end{equation}
The last inequality allows us to conclude from the induction 
hypothesis that $\omega\in L^{r}(\Lambda^{m-1}; \rho^{(m-1)} d^{m-1}x)$, and
\begin{equation}
\label{w_est}
|| \omega ||_{r,\rho^{(m-1)} d^{m-1} x} \leq \tilde{C} 
\left [ 
\binom{N}{m} \alpha^{\frac{N-m}{r}}
+1 \right ] || U ||_{r,Pd^{N}x}  
\end{equation}
for some constant $\tilde{C}=\tilde{C}(N,m-1,r,P)$.  
Using the "if" statement of the theorem one more time, we 
infer that $G_{m,m-1}\omega \in L^{r}(\Lambda^{m}; \rho^{(m)} d^{m}x)$,and $|| G_{m,m-1}\omega ||_{r,\rho^{(m)} d^{m}x}\leq ||\omega||_{r,\rho^{(m-1)} d^{m-1}x}$.  This inequality, together with (\ref{int_mg2_1}), (\ref{Uhrom_UP}), 
and (\ref{w_est}), finally give that 
$u\in L^{r}(\Lambda^{m}; \rho^{(m)} d^{m}x)$, and
\begin{equation}
\begin{split}
&|| u ||_{r,\rho^{(m)}d^{m}x} \leq \binom{N}{m} || U(\cdot,\tilde{x}_{m+1},...,\tilde{x}_N)||_{r,\rho^{(m)}dx} +\\
& || G_{m,m-1}\omega ||_{r,\rho^{(m)}dx}
\leq \left [
\binom{N}{m} \alpha^{\frac{N-m}{r}} (1+\tilde{C}) +  
\tilde{C} \right ] || U ||_{r,Pd^{N}x}.  
\end{split}
\end{equation}
\end{proof}

The next theorem shows that any symmetric probability 
density on $\Lambda^N$ can be approximated in $L^{1}(\Lambda^{N}; d^{N}x)$ by an arbitrarily close symmetric probability 
density satisfying condition (\ref{int_cond}).  Moreover, if 
measure $dx$ is finite and $P\in L^{\infty}(\Lambda^{N}; d^{N}x)$, then this approximation is in $L^{\infty}(\Lambda^{N}; d^{N}x)$.  
\begin{theor}
\label{prob_est}
If $N\geq 2$ and $P$ is a symmetric probability density on 
$\Lambda^N$, there is a sequence $(P_n)$ of 
symmetric probability densities on $\Lambda^N$ such that 
$P_n$ satisfies (\ref{int_cond}), and $P_n\rightarrow P$ in 
$L^{1}(\Lambda^N; d^{N}x)$.  If, in addition, measure $dx$ is finite, and $P$ is essentially bounded, then 
$P_n\rightarrow P$ in $L^{\infty}(\Lambda^N; d^{N}x)$.
\end{theor}
\begin{proof}
It suffices to prove the theorem when measure $dx$ is finite, and $P$ is essentially bounded. Indeed, since $(\Lambda; dx)$ is $\sigma$-finite, $\Lambda=\cup_{n=1}^{\infty} E_n$, where $|E_n|<\infty$ $\forall n$, and $E_{n}\subset E_{n+1}$.  Then, by dominated convergence, any symmetric probability density $P$ can be approximated in 
$L^{1}(\Lambda^N; d^{N}x)$ by a sequence of 
symmetric probability densities $P_n\chi_{E_n}/|| P_n\chi_{E_n}||_{1,d^{N}x}$, where $\chi_{E_n}$ is the characteristic 
function of the set $E_n$, and $P_n=\min\{P,n\}$.  

In accordance with the above comment, suppose that 
$|\Lambda|<\infty$, and $P$ is essentially bounded. 
Then, there is $c>0$ such that $\rho^{(N-1)}\leq c $ a.e. on 
$\Lambda^{N-1}$.  
If $Q_n:=\max\{P,\frac{1}{n} \}$, then $P_n:=\frac{Q_n}{|| Q_n ||_{1,d^{N}x}} $ is a symmetric probability density on 
$\Lambda^N$, and $P_n\rightarrow P$ in $L^{1}(\Lambda^N; d^{N}x)$ by dominated convergence.  In addition, since $-\frac{1}{n}\leq  P-P_n 
\leq || P ||_{\infty, d^{N}x}\left(1 -1\slash || Q_n ||_{1,d^{N}x}\right )$ on some co-null set for $n$ large enough, $P_n\rightarrow P$ in $L^{\infty}(\Lambda^N; d^{N}x)$.  

It remains to check that $P_n$ satisfies (\ref{int_cond}).  For this, we notice that 
\begin{equation}
\label{Pest_1}
P_n\geq \frac{1}{n || Q_n ||_{1,d^{N}x}} \geq 
\frac{1}{n + |\Lambda|^{N} }.
\end{equation}
Also, a.e. on $\Lambda^{N-1}$:
\begin{equation}
\label{Pest_2}
\begin{split}
&\rho_n^{(N-1)}= 
\frac{\int_{\Lambda} Q_n (\cdot,x_N)dx_N}{|| Q_n ||_{1,d^{N}x}}\leq 
\frac{\rho^{(N-1)} + \frac{1}{n} |\Lambda|}{|| Q_n ||_{1,d^{N}x}} \\
&\leq \rho^{(N-1)} + \frac{1}{n} |\Lambda| 
\leq c+ \frac{1}{n} |\Lambda|.
\end{split}
\end{equation}
From (\ref{Pest_1}) and (\ref{Pest_2}) it follows that 
for every $x_N\in \Lambda$
\begin{equation}
P_n(\cdot,x_N)\geq \alpha_n \rho_n^{(N-1)}(\cdot)~~ \text{
a.e on $\Lambda^{N-1}$},
\end{equation}
with $\alpha_n:=\left [ \left (n+ |\Lambda|^N\right )\left (c+\frac{1}{n} 
|\Lambda | \right )\right ]^{-1}$.  
Thus, (\ref{int_cond}) is satisfied by $P_n$ with 
$A=\Lambda$, and $\gamma\equiv \alpha_n$.
\end{proof}

\section*{Acknowledgements}

The author is deeply indebted to Patrick J. Rabier for his generous contributions to this work.  The author is glad to express her gratitude to William Noid for first introducing her to the inverse problem.  


%
%
%

\end{document}